\documentclass[11pt]{article}
\usepackage{latexsym,amssymb,amsmath,amsthm,enumerate,geometry,float,cite}
\newtheorem{theorem}{Theorem}

\newtheorem{lemma}[theorem]{Lemma}
\newtheorem{conjecture}[theorem]{Conjecture}

\usepackage{lineno}
\usepackage{setspace}
\allowdisplaybreaks

\begin{document}
\onehalfspace

\title{Leaf to leaf path lengths in trees of given degree sequence}
\author{Dieter Rautenbach \and Johannes Scherer \and Florian Werner}
\date{}
\maketitle
\begin{center}
{\small 
Institute of Optimization and Operations Research, Ulm University, Ulm, Germany\\
\texttt{$\{$dieter.rautenbach,johannes-1.scherer,florian.werner$\}$@uni-ulm.de}
}
\end{center}

\begin{abstract}
For a tree $T$, 
let $lp(T)$ be the number of different lengths of leaf to leaf paths in $T$.
For a degree sequence $s$ of a tree,
let ${\rm rad}(s)$
be the minimum radius of a tree with degree sequence $s$.
Recently, 
Di Braccio, Katsamaktsis, Ma, Malekshahian, and Zhao
provided a lower bound 
on $lp(T)$ in terms of the number of leaves 
and the maximum degree of $T$,
answering a related question posed by Narins, Pokrovskiy, and Szabó.
Here we show $lp(T)\geq {\rm rad}(s)-\log_2\left({\rm rad}(s)\right)$
for a tree $T$ with no vertex of degree $2$ and degree sequence $s$,
and discuss possible improvements and variants.\\[3mm]
\end{abstract}

\section{Introduction}

Let $T$ be a tree.
A vertex of degree at most $1$ in $T$ is a leaf of $T$.
A path in $T$ between leaves of $T$ is a {\it leaf to leaf path} in $T$.
Let $lp(T)$ denote the number of different lengths of leaf to leaf paths in $T$.
Let ${\rm rad}(T)$ and ${\rm diam}(T)$ denote 
the radius and the diameter of $T$, respectively.

For the construction of a counterexample 
to a conjecture of Erd\H{o}s, Faudree, Gyárfás, and Schelp \cite{erfagysc}
concerning cycle lengths,
Narins, Pokrovskiy, and Szabó \cite{naposz}
constructed trees with vertices of degrees $1$ and $3$ only
that avoid leaf to leaf paths of certain lengths.
Answering one of the related questions posed by Narins et al.~in \cite{naposz},
Di Braccio et al.~\cite{brkama} proved the following.

\begin{theorem}[Di Braccio et al.~\cite{brkama}]\label{theorem1}
If $T$ is a tree of maximum degree $\Delta$ at least $3$ with $\ell$ leaves,
then $$lp(T)\geq \log_{\Delta-1}((\Delta-2)\ell).$$
\end{theorem}
Our goal is to generalize Theorem \ref{theorem1} 
in such a way that the lower bound on $lp(T)$ 
depends on the degree sequence or the diameter of the tree $T$.
Recall that, for an integer $n$ at least $2$, 
a sequence $s=(d_1,\ldots,d_n)$ of $n$ positive integers 
is the degree sequence of some tree if and only if
$d_1+\cdots+d_n=2(n-1)$.
Typically, for a given degree sequence $s=(d_1,\ldots,d_n)$ of some tree,
there are many non-isomorphic trees with degree sequence $s$.
Let 
$${\rm rad}(s)$$ 
denote the minimum radius of a tree with degree sequence $s$.

\pagebreak

We pose the following conjecture.

\begin{conjecture}\label{conjecture1}
If $s$ is the degree sequence of a tree $T$ with no vertex of degree $2$,
then 
$$lp(T)\geq {\rm rad}(s)-O(1).$$
\end{conjecture}
Let $\Delta$ and $h$ be integers at least $3$.
If $T_{\Delta,h}$ is the tree with vertices of degrees $1$ and $\Delta$ only
in which every leaf has distance $h$ 
from some fixed root vertex $r$,
then $T_{\Delta,h}$ has $\ell=\Delta(\Delta-1)^{h-1}$ leaves.
As observed by Di Braccio et al.~\cite{brkama}, 
the trees $T_{\Delta,h}$ show that Theorem \ref{theorem1} is essentially tight,
because 
$$lp(T_{\Delta,h})=
|\{ 0,2,\ldots,2h\}|=
h+1=\left\lceil \log_{\Delta-1}((\Delta-2)\ell)\right\rceil.$$
It is easy to see that ${\rm rad}(s_{\Delta,h})=h$
for the degree sequence $s_{\Delta,h}$ of $T_{\Delta,h}$,
which means that Conjecture \ref{conjecture1}
would be best possible up to the additive constant.

If $s=(d_1,\ldots,d_n)$ is a given degree sequence of some tree
with ordered entries $d_1\geq \ldots \geq d_n$, 
then it is easy to see 
using simple exchange arguments
that ${\rm rad}(s)$ 
can be determined efficiently by a simple greedy construction
\begin{itemize}
\item creating a $0$th layer $L_0$ 
containing a center vertex $r$ of maximum degree $d_1$,
\item creating a 1st layer $L_1$ 
attaching to $r$ exactly $d_1$ neighbors of the largest possible degrees 
$d_{2},\ldots,d_{1+d_1},$
\item creating a 2nd layer $L_2$ 
attaching 
$\sum\limits_{u\in L_1}(d_T(u)-1)=(d_2-1)+\cdots+(d_{1+d_1}-1)$
vertices of the largest possible degrees 
$d_{2+d_1},\ldots,d_{1+d_1+(d_2-1)+\cdots+(d_{1+d_1}-1)}$
to the neighbors of $r$, 
\item and so forth.
\end{itemize}
The constructed tree has its vertices in layers $L_0,L_1,\ldots,L_{{\rm rad}(s)}$,
each containing all vertices at a certain distance from the root vertex
and this distance increases monotonically with the corresponding vertex degree.
The number of vertices in each layer 
depends on the degrees of the vertices in the preceding layer.
All leaves of the constructed tree 
have distance ${\rm rad}(s)$ or ${\rm rad}(s)-1$ from $r$.

If $s$ contains only entries $1$ and $\Delta\geq 3$,
then ${\rm rad}(s)$ is essentially $\log_{\Delta-1}(\ell)$,
where $\ell$ is the number of $1$ entries in $s$,
that is, Conjecture \ref{conjecture1} generalizes 
Theorem \ref{theorem1} in the desired way.
Our main result establishes Conjecture \ref{conjecture1}
up to a term of smaller order.

\begin{theorem}\label{theorem2}
If $s$ is the degree sequence of some tree $T$ with no vertex of degree $2$,
then 
$$lp(T)\geq {\rm rad}(s)-\log_2\left({\rm rad}(s)\right).$$
\end{theorem}

\section{Proof of Theorem \ref{theorem2}}

For the proof of Theorem \ref{theorem2},
it is convenient to consider rooted trees.
For notational clarity, we consider rooted trees $(T,r)$ 
as so-called {\it out-trees}, 
where all edges of $T$ are directed away from the root $r$.
If $T$ has degree sequence $(d_1,d_2,\ldots,d_n)$,
where $d_1$ is the degree of $r$, 
then the rooted tree $(T,r)$ has out-degree sequence
$(d_1^+,d_2^+,\ldots,d_n^+)=(d_1,d_2-1,\ldots,d_n-1)$.
In fact, a sequence $s^+=(d_1^+,\ldots,d_n^+)$ of non-negative integers
is the out-degree sequence of some rooted (out-)tree $(T,r)$
if and only if $d_1^++\cdots +d_n^+=n-1$.
A leaf in a rooted tree is a vertex without children.
The depth of a vertex in a rooted tree is the distance from the root to that vertex
and the height of a rooted tree is the maximum depth of a vertex.
For a rooted tree $(T,r)$ and a vertex $u$ in $T$,
the {\it rooted subtree $(T_u,u)$ of $(T,r)$ at $u$}
contains $u$ as its root as well as all descendants of $u$ in $(T,r)$.

For an out-degree sequence $s^+$ of some rooted tree, let 
$${\rm h}(s^+)$$ 
denote the minimum height of a rooted (out-)tree with out-degree sequence $s^+$.
Given $s^+$, it is easy to see that ${\rm h}(s^+)$ can be determined efficiently 
by a similar greedy construction as explained above for ${\rm rad}(s)$.
Furthermore, it is easy to see that
\begin{eqnarray}\label{e1}
{\rm h}(d_1,d_2-1,\ldots,d_n-1)={\rm rad}(d_1,d_2,\ldots,d_n)
\end{eqnarray}
for every degree sequence $s=(d_1,d_2,\ldots,d_n)$
with ordered entries $d_1\geq \ldots \geq d_n$.

For a rooted tree $(T,r)$, let $lp(T,r)=lp(T)$, 
where $T$ is the underlying undirected tree,
that is, 
we ignore the orientations for the paths lengths counted by $lp(T,r)$.

\begin{lemma}\label{lemma1}
If $(T,r)$ is a rooted tree with leaves of $k$ different depths, then 
$lp(T,r)\geq k$.
\end{lemma}
\begin{proof}
Since the statement is trivial for $k=1$, we may assume $k\geq 2$.
Let $u_1,\ldots,u_k$ be leaves of different depths in $(T,r)$.
Let $v$ be the lowest common ancestor of $u_1,\ldots,u_k$ in $(T,r)$.
Since $k\geq 2$, the vertex $v$ has at least $2$ children.
Let $w$ be a child of $v$ such that some $u_i$ is in $V(T_w)$.
Let 
\begin{eqnarray*}
A&=&\Big\{ {\rm dist}_T(v,u_i):i\in [k]\mbox{ and }u_i\in V(T_w)\Big\}\mbox{ and }\\
B&=&\Big\{ {\rm dist}_T(v,u_i):i\in [k]\mbox{ and }u_i\in V(T_v)\setminus V(T_w)\Big\}.
\end{eqnarray*}
Clearly, the sets $A$ and $B$ are non-empty and $|A\cup B|=k$.
Let 
$A=\{ a_1,\ldots,a_c\}$ with $a_1<\ldots<a_c$
and 
$B=\{ b_1,\ldots,b_d\}$ with $b_1<\ldots<b_d$.
Considering leaf to leaf paths over $v$ as well as 
one trivial leaf to leaf path of length $0$, we obtain
$lp(T,r)\geq |A+B|+1$.
Since
$A+B$ contains at least $c+d-1=k-1$ distinct sums
$$a_1+b_1<a_1+b_2<\ldots<a_1+b_d<a_2+b_d<\ldots<a_c+b_d,$$
the statement follows.
\end{proof}

Let 
$s^+=(d_1^+,\ldots,d_n^+)$ 
be the out-degree sequence of some rooted tree $(T,r)$.
Let $\ell$ be the number of $0$ entries of $s^+$,
which equals the number of leaves of $(T,r)$.
For some $k\in [\ell]$, let 
$$h(s^+,k)$$
be the minimum height of a rooted tree $(\hat{T},\hat{r})$ 
with out-degree sequence $\hat{s}^+$ such that
\begin{itemize}
\item $(\hat{T},\hat{r})$ has at least $k$ leaves and
\item $\hat{s}^+$ is a subsequence of $s^+$,
that is, the sequence $\hat{s}^+$ arises from $s^+$ by removing entries.
\end{itemize}

\pagebreak

\begin{lemma}\label{lemma2}
Given $s^+$ and $k$ as above.
\begin{enumerate}[(i)]
\item $h(s^+,k)$ is well-defined and increases monotonically in $k$.
\item $h(s^+,k)$ can be determined efficiently.
\item $h(s^+,k)\leq h\left(s^+,\left\lceil\frac{k}{2}\right\rceil\right)+1$.
\end{enumerate}
\end{lemma}
\begin{proof} (i)
Since $(T,r)$ has $\ell$ leaves and $\ell\geq k$, 
the rooted tree $(T,r)$ is a feasible choice for $(\hat{T},\hat{r})$,
which implies that $h(s^+,k)$ is well-defined.
The monotonicity follows immediately from the definition.

\medskip\noindent (ii) We may assume that $d_1^+\geq \ldots \geq d_n^+$.
The definition of $h(s^+,k)$ and simple exchange arguments 
imply that the non-zero entries of $\hat{s}^+$
can be chosen as an initial segment of $s^+$.
More precisely, there is some $p\in [n]$ with
$$\hat{s}^+=(d_1^+,\ldots,d_p^+,\underbrace{0,\ldots,0}_{\geq k}).$$
Since $\hat{s}^+$ is the out-degree sequence of some rooted tree,
the number of $0$ entries in $\hat{s}^+$
is exactly $d_1^++\cdots +d_p^+-(p-1)$,
which implies that $p$ can be chosen within $[n]$ 
as the smallest value 
for which $d_1^++\cdots +d_p^+-(p-1)\geq k$.
In particular, this choice implies that $d_p^+\geq 2$,
that is, the sequence $\hat{s}^+$ contains no $1$ entry.
Clearly, for this specific choice of $\hat{s}^+$,
we have $h(s^+,k)=h(\hat{s}^+)$,
which completes the proof of (ii).

\medskip\noindent (iii) 
We exploit the above observations.
Again, let $d_1^+\geq \ldots \geq d_n^+$.
Let $p,q\in [n]$ be smallest such that 
\begin{eqnarray*}
d_1^++\cdots +d_p^+-(p-1)&\geq& k\mbox{ and }\\
k':=d_1^++\cdots +d_q^+-(q-1)&\geq& \left\lceil\frac{k}{2}\right\rceil.
\end{eqnarray*}
Clearly, we have $p\geq q$ and $d^+_p\geq 2$.

If $k'\geq k$, then 
$$
h(s^+,k)
=h(d_1^+,\ldots,d_q^+,\underbrace{0,\ldots,0}_{=k'\geq k})
=h\left(s^+,\left\lceil\frac{k}{2}\right\rceil\right).
$$
Now, let $k'<k$.
If $k'<p-q$, then 
\begin{eqnarray*}
d_1^++\cdots +d_{q+k'}^+-(q+k'-1)
& = & 
d_1^++\cdots +d_q^+-(q-1)
+\underbrace{(d_{q+1}^+-1)}_{\geq 1}+\cdots+\underbrace{(d_{q+k'}^+-1)}_{\geq 1}
\geq 2k'\geq k,
\end{eqnarray*}
contradicting the choice of $p$.
Hence, we obtain $k'\geq p-q$.
Let $(\hat{T}',\hat{r})$ be a rooted tree of height 
$h\left(s^+,\left\lceil\frac{k}{2}\right\rceil\right)$
with $k'$ leaves and out-degree sequence 
$$(d_1^+,\ldots,d_q^+,\underbrace{0,\ldots,0}_{=k'\geq \lceil k/2\rceil}).$$
If $(\hat{T},\hat{r})$ arises from $(\hat{T}',\hat{r})$ 
by selecting $p-q$ leaves of $(\hat{T}',\hat{r})$ and
attaching to these leaves $d_{q+1}^+,\ldots,d_p^+$ new leaves, respectively,
then $(\hat{T},\hat{r})$ is of height at most
$h\left(s^+,\left\lceil\frac{k}{2}\right\rceil\right)+1$,
has $k'+d_{q+1}^++\cdots +d_p^+-(p-q)\geq k$ leaves and out-degree sequence 
$(d_1^+,\ldots,d_p^+,0,\ldots,0)$, which implies 
$h(s^+,k)\leq h\left(s^+,\left\lceil\frac{k}{2}\right\rceil\right)+1$.
\end{proof}

\begin{lemma}\label{lemma3}
If $(T,r)$ is a rooted tree with out-degree sequence $s^+$
and $U=\{ u_1,\ldots,u_k\}$ is a set of $k$ leaves of $(T,r)$ 
that are all of equal depth, 
then there are leaf to leaf paths between vertices in $U$
of at least $h(s^+,k)+1$ different lengths.
In particular, $lp(T,r)\geq h(s^+,k)+1$.
\end{lemma}
\begin{proof}
The proof is by induction on $k$.
For $k=1$, the trivial leaf to leaf path $u_1$ of length $0$ implies $lp(T,r)=1$.
Since $h(s^+,1)=0$, we have the desired inequality.
Now, let $k\geq 2$.
Let $v$ be the lowest common ancestor of $u_1,\ldots,u_k$.
Since $k\geq 2$, the vertex $v$ is not a leaf.
Let $d=d_T^+(v)$ and let $w_1,\ldots,w_d$ be the children of $v$.
For $i\in [d]$, let $(T_i,w_i)$ be the rooted subtree of $(T,r)$ at $w_i$,
let $s_i^+$ be the out-degree sequence of $(T_i,w_i)$,  
let $U_i=\{ u_1,\ldots,u_k\}\cap V(T_i)$, and 
let $k_i=|U_i|$.
Clearly, we have $k_1+\cdots+k_d=k$ and $k_i<k$ for every $i\in [d]$.
By reordering the children of $v$,
we may assume that $k_1,\ldots,k_{d'}\geq 1$
and $k_{d'+1}=\ldots=k_d=0$ 
for some $d'\in [d]\setminus \{ 1\}$.

If, for every $i\in [d']$, 
the rooted tree $(\hat{T}_i,w_i)$ 
of height $h(s_i^+,k_i)$
has at least $k_i$ leaves
and an out-degree sequence that is a subsequence of $s_i^+$,
then the rooted tree $(\hat{T},v)$ that arises 
\begin{itemize}
\item from the disjoint union of 
$v$,
$(\hat{T}_1,w_1),\ldots,(\hat{T}_{d'},w_{d'})$, and
a set $X$ of $d-d'$ further vertices
\item by adding the oriented edges 
$(v,w_1),\ldots,(v,w_{d'})$ as well as 
the oriented edges from $v$ to each element of $X$
\end{itemize}
has height
$\max\Big\{ h(s_i^+,k_i):i\in [d']\Big\}+1,$
at least $k$ leaves, and 
an out-degree sequence that is a subsequence of $s^+$.
This implies 
$$h(s^+,k)\leq \max\Big\{ h(s_i^+,k_i):i\in [d']\Big\}+1.$$
By symmetry, we may assume that $h(s^+,k)\leq h(s_1^+,k_1)+1$.
By induction,
the rooted tree $(T_1,w_1)$ contains leaf to leaf paths 
between the vertices in $U_1$ 
of $h(s_1^+,k_1)+1$ different lengths.
Since all leaves in $U$ have the same depth in $(T,r)$,
a leaf to leaf path between a vertex in $U_1$ and a vertex in $U_2$
is strictly longer than each of these paths, and we obtain 
leaf to leaf paths between vertices in $U$
of at least $h(s_1^+,k_1)+2\geq h(s^+,k)+1$
different lengths, 
which completes the proof.
\end{proof}
We are now in a position to prove Theorem \ref{theorem2}.

\begin{proof}[Proof of Theorem \ref{theorem2}]
Let $s$ and $T$ be as in the statement.
Let $s=(d_1,\ldots,d_n)$ with $d_1\geq \ldots \geq d_n$.
Rooting $T$ at a vertex $r$ of degree $d_1$ yields a rooted tree $(T,r)$
with out-degree sequence $s^+=(d_1,d_2-1,\ldots,d_n-1)$.
Let $h=h(s^+)$.
By (\ref{e1}), it suffices to show $lp(T,r)\geq h-\log_2(h)$.
Since $s$ has no $2$ entry, the sequence $s^+$ has no $1$ entry,
which easily implies $h(s^+)=h(s^+,\ell)$. 
Let the number of $0$ entries in $s^+$ be $\ell$, that is, 
the tree $T$ and the rooted tree $(T,r)$ both have exactly $\ell$ leaves. 

If $(T,r)$ has leaves of $h+1$ different depths, 
then Lemma \ref{lemma1} implies $lp(T,r)\geq h+1$.
Hence, we may assume that $(T,r)$ 
has leaves of at most $h$ different depths only.
By the pigeonhole principle, this implies that $(T,r)$
has $k\geq \frac{\ell}{h}$ leaves that are all of equal depth and 
Lemma \ref{lemma3} implies $lp(T,r)\geq h(s^+,k)+1$.
For $p=\left\lfloor\log_2\left(\frac{\ell}{k}\right)\right\rfloor$,
we obtain $\left\lceil\frac{\ell}{2}\right\rceil\leq 2^pk \leq \ell$,
and Lemma \ref{lemma2} (i) and (iii) imply
$$
h
=h(s^+)
=h(s^+,\ell)
\stackrel{(iii)}{\leq} h\left(s^+,\left\lceil\frac{\ell}{2}\right\rceil\right)+1
\stackrel{(i)}{\leq} h\left(s^+,2^pk\right)+1
\stackrel{(iii)}{\leq} h(s^+,k)+p+1$$
and, hence,
$$lp(T,r)\geq h(s^+,k)+1\geq h-p\geq h-\log_2\left(\frac{\ell}{k}\right)\geq h-\log_2(h),$$
which completes the proof.
\end{proof}

\section{Conclusion}

While it is natural to exclude vertices of degree $2$ in this context,
there is a version of Conjecture \ref{conjecture1} including them:
If $s=(d_1,\ldots,d_n)$ is the degree sequence of a tree $T$
and $s'=(d_1,\ldots,d_{n'})$ 
arises from $s$ by removing all entries equal to $2$,
then $s'$ is still the degree sequence of a tree and we conjecture 
$lp(T)\geq {\rm rad}(s')-O(1).$
Next to Conjecture \ref{conjecture1}, we pose the problem to determine 
$f(D)=\inf\{ lp(T):T\in {\cal T}(D)\},$
where ${\cal T}(D)$ is the set of all trees $T$ with no vertex of degree $2$
and diameter ${\rm diam}(T)$ equal to $D$.
Results in~\cite{brkama} imply
$lp(T)\geq \frac{{\rm diam}(T)^{2/3}}{3}$
for a tree $T$ with no vertex of degree $2$,
which implies $f(D)\geq \frac{D^{2/3}}{3}$.
We believe $f(D)=o(D)$.

We conclude with an observation related to Kraft's inequality:
If $T$ is a rooted binary tree with $\ell$ leaves
and ${\cal W}$ is the multiset of all ${\ell\choose 2}$
path lengths of non-trivial leaf to leaf paths in $T$,
then a simple inductive argument shows 
$\sum\limits_{w\in {\cal W}}2^{-w}\leq \frac{\ell-1}{4}$
with equality if and only if $T$ is a full binary tree.

\end{document}